\documentclass[a4paper,10pt]{amsart}
\usepackage[utf8]{inputenc}
\usepackage{verbatim}
\usepackage{amssymb,amsmath}
\usepackage{enumerate}
\usepackage{cite}
\usepackage[active]{srcltx}
\usepackage{mathabx}
\DeclareMathOperator{\cf}{cf}
\DeclareMathOperator{\dr}{dr}
\DeclareMathOperator{\chr}{\chr}

\numberwithin{equation}{section}
\usepackage{float}
\usepackage[left=3.5 cm, right=3.5cm]{geometry}
 \usepackage[usenames,dvipsnames]{pstricks}
 \usepackage{epsfig}
\usepackage{pst-grad}
\usepackage{pstricks,pst-node}

\usepackage{xmpmulti}
\usepackage{psfrag}
\usepackage{tikz}

\usepackage{mathtools}

\makeatletter
\def\slashedarrowfill@#1#2#3#4#5{%
  $\m@th\thickmuskip0mu\medmuskip\thickmuskip\thinmuskip\thickmuskip
   \relax#5#1\mkern-7mu%
   \cleaders\hbox{$#5\mkern-2mu#2\mkern-2mu$}\hfill
   \mathclap{#3}\mathclap{#2}%
   \cleaders\hbox{$#5\mkern-2mu#2\mkern-2mu$}\hfill
   \mkern-7mu#4$%
}

\def\rightslashedarrowfilla@{%
  \slashedarrowfill@\relbar\relbar{\raisebox{1.2pt}{$\scriptscriptstyle\diagup$}}\rightarrow}
\newcommand\xslashedrightarrowa[2][]{%
  \ext@arrow 0055{\rightslashedarrowfilla@}{#1}{#2}}

\def\rightslashedarrowfillb@{%
  \slashedarrowfill@\relbar\relbar/\rightarrow}
\newcommand\xslashedrightarrowb[2][]{%
  \ext@arrow 0055{\rightslashedarrowfillb@}{#1}{#2}}

\def\rightslashedarrowfillc@{%
  \slashedarrowfill@\relbar\relbar{\raisebox{.12em}{\tiny/}}\rightarrow}
\newcommand\xslashedrightarrowc[2][]{%
  \ext@arrow 0055{\rightslashedarrowfillc@}{#1}{#2}}

\pgfdeclareshape{slash underlined}
{
  \inheritsavedanchors[from=rectangle] 
  \inheritanchorborder[from=rectangle]
  \inheritanchor[from=rectangle]{north}
  \inheritanchor[from=rectangle]{north west}
  \inheritanchor[from=rectangle]{north east}
  \inheritanchor[from=rectangle]{center}
  \inheritanchor[from=rectangle]{west}
  \inheritanchor[from=rectangle]{east}
  \inheritanchor[from=rectangle]{mid}
  \inheritanchor[from=rectangle]{mid west}
  \inheritanchor[from=rectangle]{mid east}
  \inheritanchor[from=rectangle]{base}
  \inheritanchor[from=rectangle]{base west}
  \inheritanchor[from=rectangle]{base east}
  \inheritanchor[from=rectangle]{south}
  \inheritanchor[from=rectangle]{south west}
  \inheritanchor[from=rectangle]{south east}
  \inheritanchorborder[from=rectangle]
  \foregroundpath{
    \southwest \pgf@xa=\pgf@x \pgf@ya=\pgf@y
    \northeast \pgf@xb=\pgf@x \pgf@yb=\pgf@y
    \pgf@xc=\pgf@xa
    \advance\pgf@xc by .5\pgf@xb
    \pgf@yc=\pgf@ya
    \advance\pgf@xc by -1.3pt
    \advance\pgf@yc by -1.8pt
    \pgfpathmoveto{\pgfqpoint{\pgf@xc}{\pgf@yc}}
    \advance\pgf@xc by  2.6pt
    \advance\pgf@yc by  3.6pt
    \pgfpathlineto{\pgfqpoint{\pgf@xc}{\pgf@yc}}
    \pgfpathmoveto{\pgfqpoint{\pgf@xa}{\pgf@ya}}
    \pgfpathlineto{\pgfqpoint{\pgf@xb}{\pgf@ya}}
 }
}
\tikzset{nomorepostaction/.code=\let\tikz@postactions\pgfutil@empty}

\makeatother

\title{Balanced independent sets in graphs omitting large cliques}
\date{\today}

\author{C. Laflamme}
\address[C. Laflamme]{Mathematics \& Statistics, University of Calgary, Calgary, AB, Canada}
\email{laflamme@ucalgary.ca}

\author{A. Aranda Lopez}
\address[A. Aranda Lopez]{Mathematics \& Statistics, University of Calgary, Calgary, AB, Canada}
\email{andres.aranda@gmail.com}

\author{D. T. Soukup}
\address[D. T. Soukup]{Universit\"at Wien
Kurt G\"odel Research Center for Mathematical Logic
W\"ahringer Strasse 25
1090 WIEN
AUSTRIA}
 \email[Corresponding author]{daniel.soukup@univie.ac.at}
  \urladdr{http://www.logic.univie.ac.at/~soukupd73/}

\author{R. Woodrow}
\address[R. Woodrow]{Mathematics \& Statistics, University of Calgary, Calgary, AB, Canada}
\email{woodrow@ucalgary.ca}

\newcommand{\vareps}{\varepsilon}

\newcommand{\half}{H_{\oo,\oo}}
\newcommand{\empt}{E_{\oo,\oo}}

\newtheorem{prop}{Proposition}[section]

\newtheorem{lemma}[prop]{Lemma}

\newtheorem{cor}[prop]{Corollary}

\newtheorem{theorem}[prop]{Theorem}
\newtheorem{claim}[prop]{Claim}
\newtheorem{obs}[prop]{Observation}

\newtheorem{tclaim}{Claim}[prop]

\newtheorem{tlemma}{Lemma}[prop]

\newtheorem{prob}[prop]{Problem}
\newcommand{\mf}[1]{\mathfrak{#1}}

\newcommand{\mc}[1]{\mathcal{#1}}
\newcommand{\mb}[1]{\mathbb{#1}}

\newcommand{\oo}{\omega}

\newcommand{\omg}{\omega_1}

\newcommand{\ran}{\text{ran}}

\newcommand{\NN}{\mb N}

\newcommand{\setm}{\setminus}

\newcommand{\bij}{\hookrightarrow\mathrel{\mskip-14mu}\rightarrow}

\newcommand{\barr}{\xrightarrow[\textmd{ind}]{\textmd{bal}}}

\theoremstyle{definition}
\newtheorem{dfn}[prop]{Definition}

\makeatletter
\newtheorem*{rep@theorem}{\rep@title}
\newcommand{\newreptheorem}[2]{%
\newenvironment{rep#1}[1]{%
 \def\rep@title{#2 \ref{##1}}%
 \begin{rep@theorem}}%
 {\end{rep@theorem}}}
\makeatother

\newreptheorem{corollary}{Corollary}
\newreptheorem{theorem}{Theorem}

\begin{document}

\begin{abstract} Our goal is to investigate a close relative of the independent transversal problem in the class of infinite $K_n$-free graphs: we show that for any infinite $K_n$-free graph $G=(V,E)$ and $m\in \mb N$ there is a minimal $r=r(G,m)$ such that for any balanced $r$-colouring of the vertices of $G$ one can find an independent set which meets at least $m$ colour classes in a set of size $|V|$. Answering a conjecture of S. Thomass\'e, we express the exact value of $r(H_n,m)$ (using Ramsey-numbers for finite digraphs), where $H_n$ is Henson's countable universal homogeneous $K_n$-free graph. In turn, we deduce a new partition property of $H_n$ regarding balanced embeddings of bipartite graphs: for any finite bipartite $G$ with bipartition $A,B$, if the vertices of $H_n$ are partitioned into two infinite classes then there is an induced copy of $G$ in $H_n$ such that the images of $A$ and $B$ are contained in different classes.


\end{abstract}
\maketitle

\section{Introduction}

The initial goal of our project was to investigate the following problem: given a sparse graph with the vertices partitioned into equally large classes, can we find an independent set which meets a certain number of these classes in large sets? The well known independent transversal or `happy dean' problem (as entertainingly presented by P. Haxell \cite{haxell}) is a close relative of this question: imagine that the dean at your university is looking to form a committee so that each faculty is represented but, for the sake of reaching decisions in reasonable times, no two members of the committee hold strictly opposing opinions on certain topics. We model this problem by forming a graph with vertices corresponding to faculty members and edges connecting colleagues who cannot sit on the same committee. Now, we are looking for an independent set meeting each faculty.

The problem of finding independent transversals goes back to papers of B. Bollob\'as, P. Erd\H os, E. G. Strauss and E. Szemer\'edi in the 1970s \cite{boll1, boll2} and still is an active area of research (let us refer to \cite{haxell} again). While the strongest results for the happy dean problem come from assumptions on the maximum degree versus the number of classes (see e.g. \cite{haxellodd}), we set out to investigate infinite graphs avoiding cliques of a fixed finite size. The motivation to do this comes from a seemingly innocent conjecture of S. Thomass\'e \cite{thom}: suppose that $H_n$ is Henson's countable universal $K_n$-free graph and the vertices are partitioned into two infinite classes: red and blue. Is there an independent set which contains infinitely many red and infinitely many blue vertices at the same time? Upon answering this question, we realized that there is a rich theory of far more general results which also yield new exciting partition properties of Henson's graphs.

Let us summarize our work; the first main result of our paper is presented in Section \ref{firstsec}.

\begin{reptheorem}{ramseyupperbound}If $G=(V,E)$ is an infinite $K_n$-free graph (for some $n\in \mb N$) and $m$ is a natural number, then there is a finite $r$ such that whenever the vertices of $G$ are partitioned into $r$ sets of equal size then there is an independent set $A$ which meets at least $m$ classes in a set of size $|V|$.  
\end{reptheorem}

The minimal such $r$ will be denoted by $r(G,m)$; the above result says that given such a balanced partition, we are able to find a large independent set which meets several classes in a large set. In the proof of Theorem \ref{ramseyupperbound}, we actually bound $r(G,m)$ with a known Ramsey-number of directed graphs (denoted by $\dr(n,m)$); this bound is also shown to be tight for certain graphs.

 The finite counterpart of Theorem \ref{ramseyupperbound} is stated below.

\begin{reptheorem}{compactness}
 Suppose that $n,m\geq 2$ and $\ell \geq 1$. Then there is a finite $N=N(n,m,\ell)$ so that for every $K_n$-free graph $G$ and pairwise disjoint sets of vertices $V_i\subseteq V$ with $|V_i|\geq N$ for $i<r=\dr(n,m)$ there is an independent set $A$ so that $$|\{i<r:|V_i\cap A|\geq \ell\}|\geq m.$$
\end{reptheorem}

At this point, we don't have any information on the size of $N=N(n,m,\ell)$ since our proof is based on a compactness argument and Theorem \ref{ramseyupperbound}.

\medskip

Next, we prove general properties of the function $m\mapsto r(G,m)$ in Section \ref{propssec}. First, note that $r(G,m)$ might be defined for graphs $G$ which are not $K_n$-free for any $n\in \mathbb N$. Indeed, if all degrees are finite in an infinite graph $G$,  then $r(G,m)=m$ for all $m\geq 2$ (see Proposition \ref{findegobs}). We show various monotonicity properties of $m\mapsto r(G,m)$ and bound $r(G,m)$ using the chromatic number in Section \ref{propssec}.

Then, we proceed by calculating $r(G,m)$ for specific graphs $G$. In particular, in Section \ref{hensonsec}, we first focus on Henson's countable, universal $K_n$-free graph $H_n$: we show that $$r(H_n,m)=\dr(n,m-1)+1$$ in Theorem \ref{thomanswer}. In turn, $r(H_n,2)=2$ for all $2\leq n\in \mb N$ which answers the above cited question of S. Thomass\'e \cite[Conjecture 46]{thom}.

\medskip

Next, in Section \ref{balembedsec}, we use the equation $r(H_n,2)=2$ to deduce a new partition property of $H_n$. Recall that the graphs $H_n$ satisfy the following: whenever the vertices of $H_n$ are partitioned into $r$ classes then one can find a monochromatic copy of $H_n$; this deep result was proved for $n=3$ by P. Komj\'ath and V. R\"odl \cite{kope} and later for arbitrary $n\in \mb N$ by M. El-Zahar and  N. Sauer \cite{sauer}. In more recent developments, N. Dobrinen \cite{dobrinen} showed that $H_3$ has 'finite big Ramsey degrees'.


We apply our machinery to show that $H_n$ satisfies a strong partition property with regards to finite bipartite graphs as well.

\begin{reptheorem}{balembed}
Fix a finite bipartite graph $G$ with bipartition $A,B$. Whenever the vertices of $H_n$ are partitioned into two infinite classes then there is an induced copy of $G$ in $H_n$ such that the images of $A$ and $B$ are contained in different classes. 
\end{reptheorem}



Finally, in Section \ref{specificsec}, we look at various well-known graphs e.g. shift graphs,  unit distance graphs, and orthogonality graphs on $\mb R^n$ with the aim to calculate the exact values of $r(G,m)$. In particular, we show that determining the value of $r(G,m)$ for orthogonality graphs is equivalent to an old problem of P. Erd\H os \cite{rosenfeld, stanley, alon}: find the size of the largest set $A$ in $\mb R^n$ so that any $B\in [A]^{m+1}$ contains two perpendicular vectors. 

Our paper concludes with a list of open problems in Section \ref{probsec}.

\bigskip
\subsection{Notations} In what follows, $r$ will always denote a nonzero natural number which we also identify with the set $\{0,1\dots r-1\}$, while $\kappa$ will always stand for an infinite cardinal. We use $[X]^k$ to denote the set of $k$-element subsets of $X$. The expression $A\subseteq^*B$ means $A\setminus B$ is finite; similarly, $A=^*B$ means that $A\setm B\cup B\setm A$ is finite.

For a graph $G=(V,E)$ and $W\subseteq V$, let $G[W]$ denote the subgraph of $G$ induced by $W$. If $v\in V$ then let $N_G(v)=\{u\in V:uv\in E\}$; if it leads to no confusion we might omit the subscript $G$ and write $N(v)$ only.

We say that a partition (or colouring) $\{V_i:i<r\}$ of a set $V$ is \textit{balanced} iff every colour class has size $|V|$ if $V$ is infinite, and $||V_i|-|V_j||\leq 1$ for all $i<j<r$ if $V$ is finite.

Let $\half$ and $K_{\oo,\oo}$ denote the \emph{half graph} and complete bipartite graph on $V=2\times \NN$ i.e. $E(\half)=\{(0,k)(1,\ell):k< \ell \in \NN\}$ and $E(K_{\oo,\oo})=\{(0,k)(1,\ell):k,\ell \in \NN\}$. Let $\empt$ denote the empty bipartite graph on $2\times \NN$. 

If $G$ is any graph then let $G[A,B]$ denote the graph on $A\cup B$ with edges $\{uv:u\in A,v\in B,uv\in E(G)\}$. Note that $G[A,A]$ coincides with $G[A]$ as defined above; we will use the latter notation for subgraphs of this form. Suppose that $G$ and $H$ are graphs and $A,B\subseteq V(G)$ and $A',B'\subseteq V(H)$. We write $$H[A',B']\hookrightarrow G[A,B]$$ if there is a 1-1 graph homomorphism which maps $A'$ into $A$ and $B'$ into $B$. If this homomorphism can be chosen surjective as well, we write  $H[A',B']\bij G[A,B]$


Let $G,H$ be two graphs. Define $G\otimes H$ on vertices $V(G)\times V(H)$ and let $(u,v)(u',v')\in E(G\otimes H)$ iff $u=u'$ and $vv'\in E(H)$ or $uu'\in E(G)$. In $G\otimes H$, each subgraph induced on a set of the form $\{u\}\times H$ is isomorphic to $H$, and for any $f:V(G)\rightarrow V(H)$, the subgraph induced on $\{(u,f(u)):u\in V(G)\}$ is isomorphic to $G$. For example, $G\otimes E_\omega$ is the graph we get by blowing up the vertices of $G$ into infinite independent sets, in particular $K_2\otimes E_\omega=K_{\omega,\omega}$. Here, $K_n$ denotes the complete graph on $n$ vertices.


\section{Finding balanced independent sets in general}\label{firstsec}

Our first goal is to show that given finite $n\geq2$ and $m\geq 1$, if $G$ is an infinite $K_n$-free graph then there is a minimal number $r=r(G,m)$  so that for every balanced $r$-partition of $V(G)$ there is an independent set $A$ such that $\{i<r:|A\cap V_i|=|V|\}$ has at least $m$ elements. Recall that whenever $G$ is an infinite $K_n$-free graph then $G$ contains an independent set of size of $|V|$; indeed, this is  an easy consequence of the famous Erd\H os-Dushnik-Miller theorem \cite{kunen}: every graph on $\kappa$ many vertices either contains an infinite clique or an independet set of size $\kappa$.

We need a few definitions first.

\begin{dfn} Let $\dr(n,m)$ denote the minimal $r$ so that any directed graph on $r$ vertices contains either a transitive set of size $n$ (i.e. a set of $n$ vertices in which the edge relation is transitive), or an independent subset of size $m$. 
\end{dfn}

Note that $R(n,m)\leq \dr(n,m)\leq R(n,n,m)$ where  $R(n_0\dots n_{k-1})$ denotes the minimal $r$ so that for any colouring of the pairs of $r$ with $k$ colours, one can find a $j$-homogeneous set of size $n_j$ for some $j<k$.

The numbers $\dr(n,m)$ were introduced by A. Gy\'arf\'as \cite{gyarfas} (denoted by $R^*(n,m)$ there). In \cite{gyarfas}, certain general bounds and values of $dr(n,m)$ for small $n,m$ are calculated. For the interested reader, we cite some of these results here:

\begin{enumerate}
 \item  $\dr(n,m)\leq 2 \dr(n-1,m)+ \dr(n,m-1)-1$,
\item  $\dr(3,3)=9, \dr(3,4)=15 $,
\item $2^{(n-1)/2}\leq \dr(n,2)\leq 2^{n-1}$,
\item $3^{(n-1)/2}\leq \dr(n,n)\leq 3^{2n-2}$, and
\item $c_1\frac{m^2}{\log m}\leq \dr(3,m)\leq c_2 \frac{m^2}{\log m}$.  
\end{enumerate}

Our first goal is to prove the following.

\begin{theorem}\label{ramseyupperbound}
 Let $n,m\geq 2$ and suppose that $G$ is an infinite $K_n$-free graph. Then $r(G,m)\leq \dr(n,m)$.

\end{theorem}


We start the proof by introducing a notion of largeness very useful in our context.

\begin{dfn}
  Let $G=(V,E)$ be a graph and suppose that $A,B\subseteq V$ are of size $|V|$. We say that $A,B$ is a \textit{rich pair} (in $G$) iff $G[A',B']$ is not empty whenever $A'\subseteq A$ and $B'\subseteq B$ are of size $|V|$.
\end{dfn}

For example, the two canonical classes of $H_{\oo,\oo}$ form a rich pair in $\half$. Now, we establish a few basic properties of rich pairs which will be applied then to prove Theorem \ref{ramseyupperbound}.

\begin{lemma}\label{fatprops}
 Suppose that $G$ is a graph on $\kappa$ vertices and $A_0,A_1\in [V]^\kappa$. Then,
\begin{enumerate}
 \item either $A_0,A_1$ is a rich pair or $K_{\kappa,\kappa}\hookrightarrow G[A_0,A_1]^c$;
\item if $A_0,A_1$ is a rich pair, then so is $A_0',A_1'$ where $A_i'\subseteq A_i$ are of size $\kappa$;
\item if $A_0,A_1$ is a rich pair, then there is $i^*<2$ so that $|\{v\in A_{i^*}:|N(v)\cap A_{1-i^*}|<\kappa\}|<\kappa$. We call $A_{i^*}$ \emph{essential} in the pair $A_0,A_1$;
\item if $A_0,A_1$ is a rich pair then there are $i^*<2$ and $A_i'\subseteq A_i$ of size $\kappa$ so that for all $A_0''\subseteq A_0'$ and $A_1''\subseteq A_1$ of size $\kappa$, $A_{i^*}''$ is essential in the pair $A_0'',A_1''$. We say that $A_{i^*}'$ is the \emph{strongly essential} part of the pair $A_0',A_1'$.
\item if $A_0,A_1$ is a rich pair, $A_i$ is strongly essential and $A_i'\subseteq A_i$ of size $\kappa$ then $A_i'$ is strongly essential in the rich pair $A_0',A_1'$;
\item $K_n$ embeds into $G$ if there are sets of vertices $\{A_i:i<n\}$ of size $\kappa$ so that $A_i,A_j$ is a rich pair with $A_i$ being strongly essential for all $i<j<n$. 
\end{enumerate}

\end{lemma}
\begin{proof}
 (1) and (2) are trivially true.
\begin{enumerate}
\setcounter{enumi}{2}
\item{Suppose the statement fails. We will find $A'\subseteq A_0,B'\subseteq A_1 $ of size $\kappa$ so that $G[A',B']$ is empty.  By (2), we may assume that $|N(v)\cap A_{1-i}|<\kappa$ for all $v\in A_i$ and $i<2$. We distinguish to cases: if $\kappa$ is regular then by a straightforward transfinite induction, one picks vertices $a_\xi\in A_0$ and $b_\xi\in A_1$ so that $a_\zeta,b_\zeta\notin N(a_\xi)\cup N(b_\xi)$ if $\xi<\zeta<\kappa$. Clearly,$G[A',B']$ is empty if $A'=\{a_\xi:\xi<\kappa\}$ and $B'=\{b_\xi:\xi<\kappa\}$ which contradict richness for $A_0,A_1$.

If $\kappa$ is singular then note that for every $\lambda<\kappa$ and $i<2$ there is $X\in [A_i]^\lambda$ so that $\sup\{|N(v)\cap A_{1-i}|:v\in X\}<\kappa$. Now, by induction on $\xi< \cf(\kappa)$ we select $X_\xi\subseteq A_0$ and $Y_\xi\subseteq A_1$ so that $\sup\{|X_\xi|:\xi<\cf(\kappa)\}=\sup\{|Y_\xi|:\xi<\cf(\kappa)\}=\kappa$ and $(X_\zeta\cup Y_\zeta)\cap (N(X_\xi)\cup N(Y_\xi))=\emptyset$ for all  $\xi<\zeta<\cf(\kappa)$. We let $A'=\bigcup\{X_\xi:\xi<\kappa\}$ and $B'=\{Y_\xi:\xi<\kappa\}$.
}
\item{Suppose that the choice of $i^*=0$ and $A_i'=A_i$ fails the assumption i.e. we can find $B_i\subseteq A_i$ so that $B_0$ is not essential; so without loss of generality $|N(v)\cap B_{1}|<\kappa$ for all $v\in B_0$. Now, if the choice $i^*=1$ and the $A_i'=B_i$ fails the assumption as well then we can find $C_i\subseteq B_i$ so that $C_1$ is not essential in $C_0,C_1$; so by further shrinking $C_1$, we can suppose that $|N(v)\cap C_{0}|<\kappa$ for all $v\in C_1$. However, now $|N(v)\cap C_{i}|<\kappa$ for all $v\in C_{1-i}$ for both $i=0,1$ which contradicts (3) as $C_0,C_1$ is a rich pair.
}
\item{This  follows from the definition of being strongly essential.}
\item{Finally, we prove (6) by induction on $n$: the case $n=2$ is trivial. Suppose that $\{A_i:i<n\}$ satisfies the assumptions above and $n\geq 3$. Using the fact that $A_0$ is strongly essential in the pair $A_0,A_i$ for $1\leq i<n$, we find a $v_0\in A_0$ so that $A_i'=N(v)\cap A_i$ has size $\kappa$ for $1\leq i<n$. Note that $A_i',A_j'$ is still a rich pair with $A_i'$ being strongly essential for all $1\leq i<j<n$ by (5).  Now, apply the inductive hypothesis for $\{A_i':1\leq i<n\}$ to find $v_i\in A_i'$ so that $\{v_i:1\leq i<n\}$ induces $K_{n-1}$. Hence, $\{v_i:i<n\}$ induces $K_{n}$.}
\end{enumerate}
\end{proof}

\begin{proof}[Proof of Theorem \ref{ramseyupperbound}] 

Suppose that $r=\dr(n,m)$ and fix a balanced partition $\{V_i:i<r\}$ of a $K_n$-free graph $G$ of size $\kappa$. 

List $[r]^2$ as $\{\{i_k,j_k\}:k<N\}$ so that $i_k<j_k$. Now, define a sequence $$W^{-1}_i\supseteq W^0_i\supseteq \dots \supseteq W^{N-1}_i$$ for all $i<r$ and a function $f:[r]^2\to 3$ simultaneously as follows: first, let $W^{-1}_i$ be an infinite independent subset of $V_i$ of size $\kappa$ (this exists by the Erd\H os-Dushnik-Miller theorem).

Now, given $(W^{k-1}_i)_{i<r}$, we do the following: let $W^k_i=W^{k-1}_i$ if $i\notin\{i_k,j_k\}$. Now, consider the pair $W^{k-1}_{i_k},W^{k-1}_{j_k}$. If $W^{k-1}_{i_k},W^{k-1}_{j_k}$ is not rich then find $W^k_{i_k}\subseteq W^{k-1}_{i_k}$ and $W^k_{j_k}\subseteq W^{k-1}_{j_k}$ of size $\kappa$ so that $G[W^k_{i_k},W^k_{j_k}]$ is empty. This can be done by Lemma \ref{fatprops} (1). 

If $W^{k-1}_{i_k},W^{k-1}_{j_k}$ is rich then find $W^k_{i_k}\subseteq W^{k-1}_{i_k}$ and $W^k_{j_k}\subseteq W^{k-1}_{j_k}$ of size $\kappa$ so that either $W^k_{i_k}$ or $W^k_{j_k}$ is strongly essential in  $W^k_{i_k},W^k_{j_k}$. Finally, we define $f$ to mark these cases separately:

\begin{enumerate}[(a)]
\item if $G[W^{k}_{i_k},W^{k}_{j_k}]$ is empty, then let $f(\{i_k,j_k\})=2$;
\item if $W^k_{i_k},W^k_{j_k}$ is rich and $W^k_{i_k}$ is strongly essential, then let $f(\{i_k,j_k\})=1$;
\item if $W^k_{i_k},W^k_{j_k}$ is rich and $W^k_{j_k}$ is strongly essential, then let $f(\{i_k,j_k\})=0$.
\end{enumerate}


Finally, let $W_i=W^{N-1}_i$ for all $i<r$. Note that  $G[W_i,W_j]$ is empty if $f(i,j)=2$, otherwise $W_i,W_j$ is rich with the side marked by $f(i,j)$ being strongly essential; indeed, Lemma \ref{fatprops} (6)  implies that if we dealt with the indices $i,j$ in step $k$ (i.e. $(i,j)=(i_k,j_k)$) then at later steps, when we possibly shrank $W^k_{i_k},W^k_{j_k}$, the same side remained strongly essential.

Now, we construct a directed graph $D$ on vertices $r$ as follows: let $ij\in E$ if $i<j$ and $f(\{i,j\})=1$ and $ji\in E$ if $i<j$ and $f(\{i,j\})=0$. Otherwise, $ij$ is not an edge.  We claim that there are no transitive sets of size $n$ in $D$. Indeed, if $\{i_k:k\in I\}$ is the increasing enumeration of a transitive set then simply apply  Lemma \ref{fatprops} (6) to $\{W_{i_k}:k\in I\}$ to find a copy of $K_n$ in $G$.


So, apply $r= \dr(n,m)$: there must be an independent set $\{i_k:k\in J\}$ of size $m$ in $D$ which means that $\bigcup\{W_{i_k}:k\in J\}$ is the desired independent set in $G$.
\end{proof}

Let us show the finite counterpart of Theorem \ref{ramseyupperbound}.

\begin{theorem}\label{compactness}
 Suppose that $n,m\geq 2$ and $\ell \geq 1$. Then there is a finite $N=N(n,m,\ell)$ so that for every $K_n$-free graph $G$ and pairwise disjoint sets of vertices $V_i\subseteq V$ with $|V_i|\geq N$ for $i<r=\dr(n,m)$ there is an independent set $A$ so that $$|\{i<r:|V_i\cap A|\geq \ell\}|\geq m.$$
\end{theorem}

In other words, if $G$ is a $K_n$-free graph on at least $\dr(n,m)\cdot N(n,m,\ell)$ vertices and $\{V_i:i<\dr(n,m)\}$ is a partition of $V(G)$ with classes of size at least $N$, then there is an independent set $A$ that has at least $\ell$ elements in at least $m$ classes.

The proof follows a standard compactness argument.
\begin{proof}
 Fix $n,m\geq 2$ and $\ell \geq 1$. Suppose for a contradiction that for every $N$ there exists a $K_n$-free graph $G_N=(V^N,E^N)$ and a partition $V^N=\bigsqcup_{i<r} V^N_i$ with $|V^N_i|\geq N$ for $i<r=\dr(n,m)$ such that, given any independent set $A$, the inequality $$|\{i<r:|V^N_i\cap A|\geq \ell\}|<m$$ holds.

We may assume that $|V^N_i|=N$ and moreover that $V(G_N)=N\cdot r$ and $V^N_i=\{t\cdot r+i:t<N\}$. Take a nonprincipal ultrafilter $\mc U$ on $\NN$ and define a graph $G$ with $V(G)=\mb N$ as follows: $uv\in E(G)$ iff $$I_{uv}=\{N\in\NN: uv\in E(G_N)\}\in \mc U.$$

\begin{tclaim}
 $G$ is $K_n$-free.
\end{tclaim}
\begin{proof}
Suppose that a set of vertices $X$ induces a copy of $K_n$ in $G$. Then $I_{uv}\in\mc U$ for all $u\neq v\in X$ so  $$I=\bigcap \{I_{uv}:u\neq v\in X\}\in \mc U$$ as well; in particular, $I\neq\emptyset$. Clearly, $X$ induces a copy of $K_n$ in $G_N$ whenever $N\in I$.
\end{proof}

Let $V_i=\{t\cdot r+i:t\in \NN\}$ for $i<r$. By Theorem \ref{ramseyupperbound}, there is an independent $A^*$ and distinct $i_0\dots i_{m-1}$ so that $$|V_{i_j}\cap A^*|=\omega$$ for $j<m$. Select $A_j\in [V_{i_j}\cap A^*]^\ell$ for each $j<m$. Let $A=\bigcup\{A_j:j<m\}$.

It suffices to show the following claim in order to reach a contradiction and hence to finish the proof of the theorem.

\begin{tclaim}
 There is an $N$ so that $A$ is independent in $G_N$ and $$|\{i<r:|V^N_i\cap A|\geq \ell\}|\geq m.$$
\end{tclaim}
\begin{proof}
Note that $\mb N\setm I_{uv}\in \mc U$ for all $u\neq v\in A$ and hence $$J=\bigcap \{\mb N\setm I_{uv}:u\neq v\in A\}\in \mc U$$ as well; in particular, $J\neq \emptyset$. Clearly, $A$ is independent in $G_N$ whenever $N\in J$. Also, $A_j\subseteq V_{i_j}$ and $N\in J$ implies that $A_j\subseteq V^N_{i_j}$ and so $\ell= |A_{j}|\leq |V^N_{i_j}\cap A|$ must hold for $j<m$.
\end{proof}
This completes the proof of the theorem.

\end{proof}

Finally, we prove that the bound $\dr(n,m)$ can be attained for $K_n$-free graphs $G$.

\begin{prop}
 Suppose that $n,m\geq 2$. Then there is a $K_n$-free graph $G$ so that $r(G,m)=\dr(n,m)$.
\end{prop}
\begin{proof}
 Let $r=\dr(n,m)-1$ and let $D$ be a digraph on the vertex set $\{0,\ldots,r-1\}$ without transitive sets of size $n$ or independent sets of size $m$. Define an $r$-partite graph $G$ on classes $V_i=\{i\}\times \oo$ for $i<r$ with 
\[
\{(i,s),(j,t)\}\in E(G)\Leftrightarrow (i,j)\in E(D)\textmd{ and } s<t.
\]
Note that if $A$ is independent and meets both $V_i$ and $V_j$ in infinitely many points then $G[V_i,V_j]$ is empty and hence ${ij},{ji}$ are not edges in $D$. In turn, as $D$ has no independent sets of size $m$, we cannot find an independent set $A$ which meets $m$ classes in infinitely many points. Hence $r(G,m)>r$.

Finally, let us prove that $G$ is $K_n$-free: suppose that $v_i=(i, k_i)\in V_i$ and $\{v_i:i\in I\}$ induces a copy of $K_n$ in $G$. Note that $k_i\neq k_j$ if $i\neq j\in I$. Furthermore, $k_i<k_j$ and $(i,k_i)(j,k_j)\in E(G)$ implies that ${ij}\in E(D)$. However this contradicts the fact that $D$ has no transitive sets of size $n$.

\end{proof}

\section{General properties of $r(G,\cdot)$}\label{propssec}

Our plan is to look at the function  $m\mapsto r(G,m)$ for an arbitrary infinite $G$ and deduce a few simple properties in general. In Theorem \ref{ramseyupperbound}, we showed that $r(G,m)$ exists for all $K_n$-free graphs $G$ and for arbitrary $m\in \mb N$. However, it makes perfect sense to study $r(G,m)$  for other graphs as well given that such a value (finite or infinite) can be defined. So whenever we write $r(G,m)$ we implicitly mean that $r(G,m)$ is defined (but $G$ is not necessarily $K_n$-free for some $n$).

Let us remind the kind reader that all graphs considered are infinite in this section unless otherwise stated. The next result achieves that all trees, locally finite graphs or planar graphs satisfy $r(G,m)=m$ for all $m\geq 2$.

\begin{prop}\label{findegobs}
  Suppose that $G$ is a countable \emph{flat graph} i.e.  for every infinite set of vertices $U$ and every natural number $n$ there is a finite set  of vertices  $S$ and infinite $U'\subset U$ so that all paths connecting two elements of $U'$ of length $<n$ contain an element from $S$. Then $r(G,m)=m$ for all $m\geq 2$.
\end{prop}

Flat graphs were introduced by  K. P. Podewski and M. Ziegler \cite{ziegler}; flat graphs include all trees, locally finite graphs, planar graphs or, more generally, graphs embeddable in surfaces of finite genus. Clearly, a flat graph might contain arbitrary large finite cliques but no infinite cliques.

\begin{proof}
 Suppose that $G$ is a countable flat graph with its vertices partitioned into $m\geq 2$ balanced classes $V_0,...,V_{m-1}$. Apply the definition of flatness with $n=3$ to find infinite $V_i'\subseteq V_i$ and finite $S_i$ such that every path of length 2 between elements of $V_i'$ goes through $S_i$. We may assume that each $V_i'$ is disjoint from the union of the $S_i$. Furthermore, shrink the $V_i'$ further so that there are no edges inside $V_i'$; this can be done since $G$ contains no infinite complete subgraph so Ramsey's theorem can be applied. 

Now, we claim that each pair $V_i', V_j'$ is not rich. Indeed, take any $x \in V_i'$ and note that $N(x)\cap V_j'$ contains at most one vertex for any $i\neq j<m$; otherwise $x$ would be in a path of length 2 between vertices of $V_j'$, contradicting the disjointness with the union of the $S_i$. Therefore, the pair cannot have a half-graph as a subgraph and is not rich. Now apply Lemma \ref{fatprops} (1) $\binom{m}{2}$ times, shrinking each $V_i'$, to extract an independent set that meets all classes in infinite sets.
\end{proof}


\begin{obs}\label{indshrink}
 If $r(G,2)$ exists then every set of $|V|$ vertices contains an independent set of size $|V|$.
\end{obs}

In particular, $G$ is $K_{|V|}$-free if $r(G,2)$ exists. Let us proceed with a few observations on monotonicity.

\begin{obs} Suppose that $G$ and $H$ are graphs of size $\kappa$ and $2\leq m\in \NN$.  \label{monobs}
 \begin{enumerate}
\item $r(G,m)\leq r(G,m+1)$;
  \item If $G$ and $H$ are isomorphic modulo a set of size $<\kappa$, then $r(G,m)=r(H,m)$;
\item If $H$ is a subgraph of $G$ and $|G|=|H|$ then either $r(G,m)\geq r(H,m)$ or $r(G,m+1)\geq r(H,m)+1$; in any case, $r(G,m+1)\geq r(H,m)$.
 \end{enumerate}
\end{obs}
\begin{proof}
 (1) and (2) are trivial. 

To prove (3) suppose that $H$ is a subgraph of $G$ and $\{V_i\}_{i<r}$ is a balanced partition of $V(H)$ for $r=r(H,m)-1$ so that any independent set in $H$ meets at most $m-1$ classes in a set of size $\kappa$. If $V_r=V(G)\setm V(H)$ has fewer than $\kappa$ elements, then  $V_0\dots V_{r-2},V_{r-1}\cup V_r$ is a balanced partition of $G$ witnessing $r(G,m)>r$  i.e. $r(G,m)\geq r(H,m)$. Hence, $r(G,m+1)\geq r(H,m)$  by (1).

If $V_r$ has size $\kappa$ then $V_0\dots V_{r-2},V_{r-1}, V_r$ is a balanced partition of $G$ witnessing $r(G,m+1)>r+1$ i.e. $r(G,m+1)\geq r(H,m)+1$.
\end{proof}

Now, we present a simple idea to bound $r(G,m)$ from above.

\begin{lemma}\label{coverbound} Suppose that $G=(V,E)$ is a graph on $\kappa$ vertices, $F\in[V]^{<\kappa}$, and $V\setm F\subseteq \bigcup_{j<t}W_j$, where $|W_j|=\kappa$ and $t\in\NN$. Then $$r(G,m)\leq \sum_{j<t}(r(G[W_j],m)-1)+1 \text{ holds for all }2\leq m\in \NN.$$

\end{lemma}
\begin{proof}
Let $V\setm F=\bigcup_{j<t}W_j$ as above. Let $r=\sum_{j<t}(r(G[W_j],m)-1)+1$ and take any balanced $r$-partition of $V=\bigcup\{V_i:i<r\}$. Define $W_{j,i}= W_j\cap V_i$.

We claim that there is a $j<t$ such that $I_j=\{i<r:|W_{i,j}|=\kappa\}$ has at least $r(G[W_j],m)$ elements. Indeed, for any $i<r$ there is $j<t$ so that $W_{j,i}$ has size $\kappa$ since $V_i\subseteq \bigcup_{j<t}W_j\cup F$ and $|F|<|V_i|$. So $r\subseteq \bigcup_{j<t} I_j$. In turn, if $|I_j|\leq r(G[W_j],m)-1$ for all $j<t$ then $r\leq \sum_{j<t}(r(G[W_j],m)-1)$, contradicting the definition of $r$.

Now, suppose that $I_{j_0}=\{i_k:k<\ell\}$ contains at least $r(G[W_{j_0}],m)$ elements. Let $X=W_{j_0}\setm \bigcup_{k<\ell}W_{j_0,i_k}$. Note that $|X|<\kappa$ and $$W_{j_0}=W_{j_0,i_0}\cup\dots W_{j_0,i_{\ell-2}}\cup (W_{j_0,i_{\ell-1}}\cup X)$$ is a balanced partition of $W_{j_0}$ into $\ell$ pieces, so there must be an independent set $A\subseteq W_{j_0}$ which meets at least $m$ pieces in a set of size $\kappa$.  As $|X|<\kappa$, $A$ must meet at least $m$ of the sets $W_{j,i_k}\subseteq V_{i_k}$.
\end{proof}

\begin{cor}\label{chromcor}
 Suppose that $G$ is a graph with finite chromatic number $\chi(G)$. Then $$r(G,m)\leq \chi(G)\cdot(m-1)+1.$$
\end{cor}
\begin{proof}
 Simply note that $r(E_{\kappa},m)=m$ where $E_\kappa$ is the empty graph on $\kappa$ vertices and apply Lemma \ref{coverbound}. 
\end{proof}

Note that the above argument actually gives $$r(G,m)\leq \min\{\chi(G[V\setm F]):F\in [V]^{<\kappa}\}(m-1)+1.$$

\begin{cor}\label{prodcor}
 Suppose that $G$ is finite graph on $N$ vertices and $H$ is arbitrary. Then $$r(G\otimes H,m)\leq N\cdot (r(H,m)-1)+1.$$
\end{cor}
\begin{proof}
 Indeed, $G\otimes H$ is covered by $N$-many copies of $H$ and so Lemma \ref{coverbound} can be applied.
\end{proof}

\begin{cor}\label{halfcor}
For all $n\geq1$ and $m\geq2$, $r(K_n\otimes E_\omega,m)=n(m-1)+1$ and  $r(\half,m)=2(m-1)+1$. Moreover,  if $A,B$ is a rich pair in a graph $G$ then $$r(G[A,B],m)=2(m-1)+1.$$
\end{cor}
\begin{proof}
First, $r(K_n\otimes E_\omega,m)\leq n(m-1)+1$ follows from Corollary \ref{prodcor}. On the other hand, if we partition each canonical class of $K_n\otimes E_\omega$ into $m-1$ infinite pieces then we get a partition of $K_n\otimes E_\omega$ into $n(m-1)$ independent sets so that no independent set $A$ intersects $m$ different pieces, so $n(m-1)+1\geq r(K_n\otimes E_\omega,m)$. The same argument works for $\half$ and the rich pair.
\end{proof}

Next, we show a somewhat surprising property of the function $m\mapsto r(G,m)$.

\begin{theorem}
If $r(G,m)=m$ for any $m\geq 3$ then $r(G,m)=m$ for all $m\geq 2$.
\end{theorem}
\begin{proof}
 
Let us start with a lemma about rich pairs.

\begin{lemma}\label{fatlemma}
 If $G=(V,E)$ has no rich pairs and $r(G,2)$ exists then $r(G,m)=m$ for all $m\geq 2$.

\end{lemma}
\begin{proof} 
Suppose that $\{V_i\}_{i<m}$ is a balanced partition. Find independent $V_i'\in[V_i]^{|V|}$ for each $i<m$; this can be done by Observation \ref{indshrink}. Apply the fact that $G$ has no rich pairs $\binom{m}{2}$-times to find $W_i\subseteq V_i'$ so that there is no edge between $W_i$ and $W_j$ if $i<j<m$. Now, $\bigcup_{i<m}W_i$ is the desired independent set.
\end{proof}

Finally suppose, that $r(G,m)=m$ for some $m\geq 3$. We claim that $G$ cannot have any rich pairs and hence we are done by Lemma \ref{fatlemma}. Indeed, suppose that $A,B$ is a rich pair; then $$m=r(G,m)\geq r(G[A,B],m-1)= 2(m-2)+1=2m-3$$ by Observation \ref{monobs} (3) and hence $m=3$. If $V\setm (A\cup B)$ has size $\kappa$ then $A,B,V\setm(A\cup B)$ is a balanced partition witnessing $r(G,3)>3$; a contradiction. If $V\setm (A\cup B)$ has size $<\kappa$ then $3=r(G,3)=r(G[A,B],3)=5$ (a contradiction again).
\end{proof}

As a trivial example, the fact that any finitely-branching tree $T$ has $r(T,m)=m$ for all $m\geq 2$ follows from Lemma \ref{fatlemma}. The same holds for trees of cardinality $\kappa$ such that each vertex has degree at most $\lambda$ for some $\lambda<\kappa$.

We showed that the reason $r(G,m)$ is bigger than $m$ for any $m$ is because there is a rich pair in $G$ in which case $r(G,m)\geq 2m-3$ for all $2\leq m\in \NN$. 

Finally, let us prove that rich pairs in countable graphs are rather easily detected; this result will be applied in the next section as well.

\begin{lemma} \label{halfembed}
 If $A,B$ is a rich pair in a countable graph $G$ then $H_{\omega,\omega}\hookrightarrow G[A,B]$, or $H_{\omega,\omega}\hookrightarrow G[B,A]$.
\end{lemma}
\begin{proof}
 We define disjoint finite $E_0,E_1\dots \subseteq A$ and $F_0,F_1\dots \subseteq B$ along with  infinite $A=A_{-1}\supseteq A_0\supseteq A_1\dots$ and $B=B_{-1}\supseteq B_0\supseteq B_1\dots$ as follows:
\begin{enumerate}[(i)]
 \item $G[E_n,F_n]$ is independent and $E_n\cup F_n\neq \emptyset$,
\item $E_n\cap A_n=\emptyset$, $F_n\cap B_n=\emptyset$,
\item if $E_n\neq \emptyset$ then there is $u_n\in E_n$ such that $B_n\subseteq N(u_n)\cap B_{n-1}$,
\item if $F_n\neq \emptyset$ then there is $v_n\in F_n$ such that $A_n\subseteq N(v_n)\cap A_{n-1}$.
\end{enumerate}
Given $A_n$ and $B_n$ we inductively select distinct $x_0\in A_n$, $y_0\in B_n$, $x_1\in A_n$, $y_1\in B_n$\dots so that $G[\{x_k:k<n\}\{y_k:k<n\}]$ is the empty graph and $A_n\setm \bigcup \{N(y_k):k<n\}$ and $B_n\setm \bigcup \{N(x_k):k<n\}$ are both infinite. This process must stop at some point as $A,B$ is a rich pair. If $n$ is minimal so that we can't choose $x_n$ then $N(x)\cap B_n$ is infinite for any $x\in A_n\setm \{x_k:k<n\}$. We pick any $u_{n+1}\in A_n\setm \{x_k:k<n\}$ and let $E_{n+1}=\{u_{n+1}\}$ and $F_{n+1}=\emptyset$, and we define $B_{n+1}=N(u_{n+1})\cap B_n$ and $A_{n+1}=A_n\setm E_{n+1}$.

If we can choose $x_n$ but $n$ is minimal so that we can't choose $y_n$ then $N(y)\cap A_n$ must be infinite for any $y\in B_n\setm \{x_k:k<n\}$. We finish the proof as before but now $F_{n+1}\neq\emptyset$.

Suppose we defined these sequences. If $E_n\neq \emptyset$ for infinitely many $n$ then $H_{\omega,\omega}\hookrightarrow G[A,B]$ and if $F_n\neq \emptyset$ for infinitely many $n$ then $H_{\omega,\omega}\hookrightarrow G[B,A]$.

\end{proof}

The case for uncountable graphs is much more subtle: let $f:[\omega_1]^2\to 2$ be J. Moore's L-space colouring \cite{Lspace} and define a bipartite graph $G$ by letting $V(G)=\omega_1\times 2$ with $(\alpha,i)(\beta,j)\in E(G)$ iff $\alpha<\beta$, $i=0,j=1$ and $f(\alpha, \beta)=1$. Recall that $f$ has the property that whenever $X,Y$ are uncountable subsets of $\omg$ and $i<2$ then there is $\alpha\in X, \beta\in Y$ so that $\alpha<\beta$ and $f(\alpha,\beta)=i$. Hence, the half graph on $\omg$ does not embed into $G$ while $A,B$ is still a rich pair.

\section{Henson's $H_n$ and the values $r(H_n,m)$}\label{hensonsec}

The first specific graphs we look at in detail are $H_n$: the countable, universal homogeneous $K_n$-free graphs defined by Henson \cite{henson}. Aside from $H_n$ being homogeneous, we mention the following properties for future reference:

\begin{enumerate}
\item \label{ext} \emph{extension property:} for any disjoint finite sets of vertices $A,B\subset H_{n}$ where $B$ is $K_{n-1}$-free, there is a vertex $v$ so that $A\cap N(v)=\emptyset$ and $B\subseteq N(v)$  \cite{henson};
 \item \label{partprop}\emph{indivisibility}: $H_n\to (H_n)^1_r$ for any $r<\omega$ \cite{kope,sauer};
\item \label{edge} if $vw\notin E(H_n)$ then $v$ and $w$ has infinitely many common neighbours;
\item \label{nbhd} if $v\in V(H_{n+1})$ then the graph induced by $N_{H_{n+1}}(v)$ in $H_{n+1}$ is isomorphic to $H_n$.
\end{enumerate}

We remark that $H_n$ is the unique countable graph satisfying property (\ref{ext}) \cite{henson}. Let us state a simple lemma as well.

\begin{lemma}\label{coverclaim0}
 The vertices of $H_{n+1}$ are not covered by finitely many $K_n$-free induced subgraphs.
\end{lemma}
\begin{proof}
Suppose that $f:V(H_{n+1})\to r$ is a finite colouring so that $H_{n+1}[f^{-1}(i)]$ is $K_n$-free. Now, by indivisibility of $H_{n+1}$, (property (\ref{partprop}) above), there is a monochromatic copy of $H_{n+1}$. But this is impossible since $H_{n+1}$  is not $K_n$-free.
\end{proof}

\begin{obs}
 Let $G$ be a countably infinite $K_n$-free graph. Then $$r(G,m)\leq r(H_n,m+1)-1.$$
\end{obs}
\begin{proof}
 Suppose that $G$ is an infinite $K_n$-free graph with a balanced partition $\{V_i:i<r\}$ where $r=r(H_n,m+1)-1$. Using the universality of $H_n$, embed $G$ into $H_n$ with a map $f$ so that $V(H_n)\setm \ran (f)$ is infinite (this is possible by indivisibility). 

Let $W_i=f[V_i]$ for $i<r$ and $W_r=V(H_n)\setm \ran (f)$. Now, $\{W_i:i\leq r\}$ is a balanced $r+1=r(H_n,m+1)$-partition of $H_n$ so there is an independent set $A$ such that $\{i\leq r:|W_i\cap A|=\omega\}$ has at least $m+1$ elements. Hence $$|\{i< r:|W_i\cap A|=\omega\}|\geq m.$$

Now $B=\bigcup\{f^{-1}(A\cap W_i):i<r\}$ is the independent set of $G$ which meets at least $m$ classes of the orginal partition $\{V_i:i<r\}$.
\end{proof}


Next, we determine the exact values $r(H_n,m)$ for all $m$ using the Ramsey numbers $\dr(n,m)$. In particular, we show $r(H_n,2)=2$ for all $n\geq 3$ which answers Conjecture 46 of Thomass\'e \cite{thom}.

\begin{theorem}\label{thomanswer}
 $ r(H_n,m)=\dr(n,m-1)+1$ for all $n,m\geq2$.
\end{theorem}

Now, using that $\dr(n,1)=1$, the next corollary is immediate.

\begin{cor} $ r(H_n,2)=2$ for all $n\geq 2$.
\end{cor}

We mention here that Theorem \ref{ramseyupperbound}  for countable graphs $G$ is now an easy corollary of Theorem \ref{thomanswer}: $r(H_n,m)=\dr(n,m-1)+1$ together with the above observation on $r(G,m)\leq r(H_n,m+1)-1$ yields $r(G,m)\leq \dr(n,m)$. 

\medskip
Now, we prove Theorem \ref{thomanswer}; recall that $\alpha(D)=\sup\{|A|:A\subseteq V(D) \text{ is independent}\}$. 

\begin{lemma} \label{examples2}
Suppose $n\geq 3$ and $D$ is a finite digraph on $r$ vertices. If $D$ has no transitive sets of size $n$, then there is a balanced partition $\bigcup\{V_i:i< r+1\}$ of $H_n$ such that any independent set $A$ meets at most $\alpha(D)+1$ members of the partition in an infinite set.
\end{lemma}


 \begin{proof}
Suppose that $D$ has vertices $\{w_i:i<r\}$ and consider the graph $G$ on vertices $\bigcup \{\{w_i\}\times \NN:i<r\}$ inducing the  half-graph $H_{\oo,\oo}$ on each pair of classes corresponding to edges in $D$ as follows: $(w_i,k)(w_j,\ell)\in E(G)$ iff ${ij}\in E(D)$ and $k<\ell\in \NN$. Note that $G[\{w_i\}\times \NN]$ is empty for all $i<r$. 

\begin{tclaim}
 $G$ is $K_n$-free.
\end{tclaim}
\begin{proof}
 
Indeed, suppose that $\{(w_i,k_i):i\in I\}$  is a copy of $K_n$ for some $I\subseteq r$ and $k_i\in \NN$. Note that $k_i\neq k_j$ if $i\neq j\in I$. Furthermore, $k_i<k_j$ and $(w_i,k_i)(w_j,k_j)\in E(G)$ implies that ${ij}\in E(D)$. However this contradicts the fact that $D$ has no transitive sets of size $n$.

\end{proof}

Now, $G$ embeds into $H_n$ and we identify $G$ and its copy in $H_n$; using the indivisibility of $H_n$, we can suppose that $V(H_n)\setm V(G)$ is infinite. Consider the $r+1$-partition $V(H_n)\setm V(G),\{w_0\}\times \NN \dots \{w_{r-1}\}\times \NN$ of $V(H_n)$. If $A$ is an independent set then $J=\{i<r: |A\cap (\{w_i\}\times \NN)|= \oo\}$ has size at most $\alpha(D)$. Indeed, if $i\neq j\in J$ then ${ij},{ji}\notin E(D)$, and hence $\{w_i:i\in J\}$ is an independent set in $D$.

So $A$ meets at most $\alpha(D)+1$ members of the partition in an infinite set as required.
 \end{proof}

 \begin{proof}[Proof of Theorem \ref{thomanswer}]
First, we show $r(H_n,m)\geq \dr(n,m-1)+1$. Let $D$ be a digraph on $r=\dr(n,m-1)-1$ vertices without transitive sets of size $n$ or independent sets of size $m-1$ i.e. $\alpha(D)\leq m-2$. Now apply Lemma \ref{examples2} to find a partition of $H_n$ into $r+1=\dr(n,m-1)$ classes so that evey independent set is contained in at most $\alpha(D)+1\leq m-1$ classes. This partition witnesses $r(H_n,m)\geq \dr(n,m-1)+1$.

Now, we prove that $r(H_n,m)\leq \dr(n,m-1)+1$. Let $\{V_i:i\leq r\}$ denote a balanced partition of $H_n$ where $r=\dr(n,m-1)$. We can suppose that there is $W\subseteq V_r$ so that $H_n[W]$ is isomorphic to $H_n$ by indivisibility. Now, find infinite $W_r\subseteq W$ and $W_i\subseteq V_i$ for $i<r$ so that there are no edges from $W_i$ to $W_r$; this can be done by picking vertices and applying the next claim.

\begin{claim}\label{coverclaim}
 $W\setm \bigcup\{N(v):v\in F\}$ is infinite for all finite $F\subseteq V(H_n)$.
\end{claim}
\begin{proof} Indeed, this follows from Lemma \ref{coverclaim0} and the fact that $H_n[N(v)\cap W]$ is a subgraph of $H_{n-1}$ by property (\ref{nbhd}).

\end{proof}

By shrinking each $W_i$, we can suppose that $H_n[W_i]$ is empty for $i\leq r$. By successively applying Lemma \ref{halfembed} and shrinking $W_i$ for $i<r$, we can suppose that either $H_n[W_i,W_j]$ is empty or $\half\bij H_n[W_i,W_j]$ or $\half\bij H_n[W_j,W_i]$ for all $i<j<r$.

Next, define a digraph $D$ on $r$ so that $ij\in E(D)$ iff $\half\bij H_n[W_i,W_j]$ for all $i\neq j<r$. As $D$ has $\dr(n,m-1)$ many vertices, we can either find an independent set of size $m-1$ or a transitive set of size $n$. As the second alternative must fail by Lemma \ref{fatprops} (6), there is $I\subseteq r$ of size $m-1$ so that $H_n[W_i,W_j]$ is empty if $i\neq j\in I$. Hence, $\bigcup\{W_i:i\in \{r\}\cup I\}$ is the desired independent set.
 \end{proof}




Next, we prove a result, one that also implies $r(H_n,2)=2$, which will be applied in the proof of Theorem \ref{balembed} later.

\begin{theorem}\label{indlemma}
Fix $n\geq 2$ and let $V_0\cup V_1$ be a balanced partition of the vertices of $H_{n+1}$. Then there is an induced copy of $H_{n}$ intersecting both $V_0$ and $V_1$ in an infinite set.

\end{theorem}

It is clear that $r(H_n,2)=2$ follows by induction on $n$.

\begin{proof}[Proof of Theorem \ref{indlemma}] 
 Let  $V_0\cup V_1$ be a balanced partition of the vertices of $H_{n+1}$ for some $n\geq 2$.
Recall that $N(v)=N_{H_{n+1}}(v)$ induces a subgraph isomorphic to $H_n$ for every vertex $v$. So, without loss of generality, we can suppose that there is $j_v<2$ so that  $N(v)\subseteq^* V_{j_v}$ for every vertex $v$. 

\begin{tclaim}\label{obs}
 If $uv\notin E=E(H_{n+1})$ then $j_v=j_u$.
\end{tclaim}
\begin{proof}
Indeed, if $uv\notin E$ then $N(u)\cap N(v)\subseteq^* V_{j_v}$ is infinite by (\ref{edge}) and hence $N(u)\cap V_{j_v}$ is infinite as well. In this case, $N(u)\subseteq^* V_{j_u}$ and so $j_v=j_u$.\end{proof}

Now, let $i_v<2$ denote the class of $v$ i.e. $v\in V_{i_v}$. The map $v\mapsto (i_v,j_v)$ is a $4$-colouring of the vertices of $H_{n+1}$ and so, using the indivisibility of $H_{n+1}$, we can find a set of vertices $W_0$ and $(i,j)\in 2\times 2$ so that $(i_v,j_v)=(i,j)$ for all $v\in W_0$ and $W_0$ induces a copy of $H_{n+1}$. Note that $i\neq j$ would imply that every vertex in $W_0$ has finite degree in $H_{n+1}[W_0]$ which contradicts that $H_{n+1}[W_0]$ is a copy of $H_{n+1}$. Hence $i=j$ and let us suppose that this common value is 0. 


\begin{tclaim}
 $N(u)\subseteq^* V_0$ for almost every vertex $u\in V_1$ and every $u\in V_0$.
\end{tclaim}
\begin{proof}
We would like to show first that $j_u=0$ for almost every $u\in V_1$. Fix an arbitrary $v\in W$. Then $j_v=0$ so $uv\notin E(H_n)$ for almost every  $u\in V_1$. Hence,  $j_u=0$  for almost every  $u\in V_1$ by Claim \ref{obs}.

Now, take $u\in V_0$ and suppose that $N(u)\subseteq^* V_1$ i.e. $j_u=1$ to reach a contradiction. If $v\in W_0$ then $uv\in E$ (otherwise  $j_v=j_u$). So $W_0\subseteq N(u)$ which contradicts $j_u=1$.   

\end{proof}

Without loss of generality, we can assume that $N(u)\subseteq^* V_0$ for every vertex $u\in V_1$.

\begin{tclaim}\label{selection}\leavevmode
\begin{enumerate}
               \item $V_1\setm \bigcup\{N(v):v\in F\}$ is infinite for any finite set of vertices $F$.
\item $V_0\cap \bigcap\{N(v):v\in F_0\}\setm \bigcup\{N(v):v\in F_1\}$ is infinite for any finite, nonempty $F_0$ which induces a $K_{n-1}$-free subgraph and any finite $F_1$.
              \end{enumerate}

\end{tclaim}
\begin{proof} (1) $N(v)\subseteq^* V_0$ implies that $V_1\cap N(v)$ is finite so $V_1\setm \bigcup\{N(v):v\in F\}$ is infinite for any finite set of vertices $F$.

(2) Note that $\bigcap\{N(v):v\in F_0\}\setm \bigcup\{N(v):v\in F_1\}$ is infinite by the extension property (\ref{ext}) and that 
$$\bigcap\{N(v):v\in F_0\}\setm \bigcup\{N(v):v\in F_1\}\subseteq N(u)\subseteq^* V_0$$ for any $u\in F_0$. This proves that $V_0\cap \bigcap\{N(v):v\in F_0\}\setm \bigcup\{N(v):v\in F_1\}$ is infinite.
%
%
\end{proof}

Now, take an enumeration $x_0,x_1\dots$ of the vertices of $H_n$ so that $I=\{i\in \mb N:x_j\notin N(x_i)$ for all $j<i\}$ is infinite. It suffices to construct an embedding $f:H_n\to H_{n+1}$ as $x_i\mapsto y_i$ so that $i\in I$ if and only if $y_i\in V_1$.

Let $f(x_0)=y_0\in V_1$ arbitrary. Now, given $y_i$ for $i<k$, we consider two cases: if $k\in I$ then simply find $y_k\in V_1\setm \bigcup\{N(y_i):i<k\}$ so that $y_k\neq y_i$ for $i<k$ by applying Claim \ref{selection} (1). If $k\in \mb N\setm I$ then let $F_0=\{y_i:i<k,x_i\in N(x_k)\}$ and $F_1=\{y_i:i<k,x_i\notin N(x_k)\}$. Now, find $y_k\in V_0\cap \bigcap\{N(v):v\in F_0\}\setm \bigcup\{N(v):v\in F_1\}$ so that $y_k\neq y_i$ for $i<k$ by applying Claim \ref{selection} (2).


\end{proof}

Finally, let us mention that any analogue of this statement for the Rado graph $R$ fails, and in particular $r(R,2)$ does not exist.

\begin{prop}
 There is a balanced 2-partition of the vertices of the Rado graph $R$ such that any infinite independent or infinite complete subgraph is monochromatic modulo a finite set.
\end{prop}

\begin{proof}
 We start from the half graph $G_0=(\omega\times \{0\}\cup \omega\times \{1\},E)$ where $$\{(n,i),(m,j)\}\in E \iff i=0,j=1 \text{ and } n<m<\omega.$$

Let $V=\omega\times \{0\}\cup \omega\times \{1\}$ and $V_i=\omega\times \{i\}$. Let us enumerate all pairs $(a,b)$ of finite subsets of $V$ as $\{(a_n,b_n):n<\omega\}$. Let $$m_n=\max\{k:(k,i)\in a_n\cup b_n\text{ for some } i\}+m_{n-1}.$$

Now define $G$ on vertices $V$ so that $$E(G)=E(G_0)\cup \{\{(m_n,0),v\}:v\in a_n,n<\omega\}.$$ 

First, note that $G$ satisfies Rado's extension property and so $G$ is the Rado graph. Second, the partition $V_0\cup V_1$ witnesses the theorem; indeed, $G[V_1]$ is independent so every complete subgraph intersects $V_1$ in at most one vertex. On the other hand, if $A$ is an independent set and $(n,0)\in A$ then $A\cap V_1\subseteq \{(k,1):k<n\}$.

\end{proof}

\section{Balanced embeddings in $H_n$}\label{balembedsec}

The motivation for the next results comes from the following question: if $G$ is countable and $K_n$-free then $G$ embeds into $H_n$ as an induced subgraph, but how much controll do we have over the  embedding? In particular, given a bipartite graph $G[A,B]$, can we ensure that the classes $A,B$ go into prescribed sets in $H_n$?

We start with a simple result in this direction.

\begin{claim}\label{weakembed} Given $n\geq3$ and a countable $K_n$-free $G$ with an arbitrary partition $A,B$ there is a balanced partition $V_0,V_1$ of $H_n$ so that  $G[A,B]\hookrightarrow H_n[V_0,V_{1}]$ as an induced subgraph (so we can require the graph homomorphism to also preserve the non-edge relation). 
\end{claim}
\begin{proof}Extend $V(G)$ by an infinite set of new vertices $W=\{v_\ell:\ell\in \NN\}$. List all pairs $(a,b)$ of finite subsets of $V=W\cup V(G)$ as $\{(a_k,b_k):k\in \NN\}$. Inductively add edges as follows: at step $k$, if $b_k$ is $K_{n-1}$-free then take a so far isolated vertex $v_{\ell_k}\in W$ and connect with all points in $b_k$. This process guarantees that the graph spanned by $a_k\cup b_k\cup \{v_{\ell_k}\}$ does not change after step $k$ and, after $\omega$ steps, we have a graph on $V$ satisfying the extension property (\ref{ext}) of $H_n$. Hence this graph is isomorphic to $H_n$. Finally, we let $V_0=W\cup A$ and $V_1=B$. 
\end{proof}

Now, we are interested if the following stronger property is satisfied: fix a graph $G$, subsets $A,B$ of the vertex set of $G$, and an \emph{arbitrary balanced partition} $V_0,V_{1}$ of $H_n$. Is there an $i<2$ so that $G[A,B]\hookrightarrow H_n[V_i,V_{1-i}]$ as an induced subgraph? If the answer is yes, then we will write   $$H_n\barr(G[A,B])^1_2$$ while the negation will be denoted by $H_n\xslashedrightarrowa[\textmd{ind}]{\textmd{bal}}(G[A,B])^1_2$. We will omit the mention of the partition of $G$ when it is clear from the context or unique.

For example, $r(H_n,2)=2$ is equivalent to $H_n\barr(E_{\oo,\oo})^1_2$. In Theorem \ref{indlemma}, we strengthened this by proving that $$H_n\barr(H_{n-1}[A,B])^1_2$$ holds for all $n\geq 3$ and a particular partition $A,B$ of the vertex set of $H_{n-1}$.

There are some obvious limitations on the type of results we can hope to prove concerning the partition relation $H_n\barr(G)^1_2$.



\begin{obs}$H_n\xslashedrightarrowa[\textmd{ind}]{\textmd{bal}}(K_{\oo,\oo})^1_2$.
\end{obs}
\begin{proof} We need to construct a partition of $H_n$ to witness $H_n\xslashedrightarrowa[\textmd{ind}]{\textmd{bal}}(K_{\oo,\oo})^1_2$. Apply the proof of Claim \ref{weakembed} starting with $G=E_{\oo,\oo}$. The inductive construction carried out there gives a balanced partition of $H_n$ with $V_0,V_1$ so that $N(v_0)\cap V_1$ is finite for all $v\in V_0$. Hence, any copy of $K_{\oo,\oo}$ is modulo finite contained in $V_0$.
\end{proof}

Let us also remark that the above partition shows why we allow embeddings into $H_n[V_0,V_{1}]$ and $H_n[V_1,V_{0}]$ at the same time. Indeed, $\half\not\hookrightarrow H_n[V_0,V_{1}]$ in the previous example but $\half\hookrightarrow H_n[V_1,V_{0}]$. In fact, the following result holds.

\begin{theorem}\label{halfbalembed}
$$H_n\barr(\half)^1_2$$ for all $n\geq 3$. 

\end{theorem}
\begin{proof} Fix a balanced partition $V_0,V_1$ of $H_n$. By $r(H_n,2)=2$, there are infinite $X=\{x_k:k\in\NN\}\subseteq V_0$, $Y=\{y_k:k\in\NN\}\subseteq V_1$ so that $X\cup Y$ is independent. Let $F_k=\{x_\ell,y_\ell:\ell\leq k\}$ and note that for every $k\in \NN$ there is a $j_k\in 2$ so that $H_n$ embeds into $N[F_k]\cap V_{j_k}$; here, $N[F]=\bigcap\{N(v):v\in F\}$. In particular, there is a single $j\in 2$ and infinite $I\subseteq \NN$ so that $j_k=j$ whenever $j\in I$. Without loss of generality, we assume $j=1$.

Select a decreasing sequence $W_k\subseteq N[F_k]\cap V_{1}$ so that $H_n[W_k]$ is isomorphic to $H_n$. First, try to select $k_0,k_1\dots \in I$ and $w_0\in W_{k_0}, w_1\in W_{k_1}\dots$ so that $\{x_{k_i},w_i:i\in \NN\}$ induces a copy of $\half$. We do this while making sure that $\{x_k:k\in I\}\setm \bigcup \{N(w_{i'}):i'< i\}$ is infinite which ensures that the next $x_{k_i}$ can be selected.

 Given $x_{k_0},w_0\dots x_{k_i}$ note that $W_{k_i}\setm \bigcup \{N(w_{i'}):i'<i\}$ still contains a copy of $H_n$ by Lemma \ref{coverclaim0}. So, if we can find $w_i\in W_{k_i}\setm \bigcup \{N(w_{i'}):i'<i\}$ so that $\{x_k:k\in I\}\setm \bigcup \{N(w_{i'}):i'\leq i\}$ is still infinite then we can continue to select $x_{k_{i+1}}$ and we construct the desired copy of $\half\hookrightarrow H_n[V_0,V_1]$.

Otherwise, there is some $i$ and a copy $W\subseteq V_1$ of $H_n$ so that the infinite independent set $A=\{x_k:k\in I\}\setm \bigcup \{N(w_{i'}):i'< i\}$ is modulo finite covered by $N(w)$ whenever $w\in W$. We claim that $\half\hookrightarrow H_n[V_1,V_0]$ holds in this case.

Indeed, start selecting distinct $w_0 \in W, v_0\in A,w_1\in W,v_1\in A\dots $ so that 

\begin{equation}\label{choice1}
v_k\in \bigcap\{ N(w_\ell):\ell\leq k\}\setm \{v_\ell:\ell<k\}
\end{equation}
and 
\begin{equation}\label{choice2}
 w_{k+1}\in W\setm \bigcup \{N(w_\ell),N(v_\ell),\{w_\ell\}:\ell \leq k\}.
\end{equation}
Note that (\ref{choice1}) is possible as $A\subseteq ^* N(w_\ell)$ and (\ref{choice2}) is possible by Lemma \ref{coverclaim0} and the fact that $W$ is a copy of $H_n$. Now, $\{w_k,v_k:k\in \NN\}$ is the desired copy of $\half$.

\end{proof}

The main result of this section is

\begin{theorem}\label{balembed}
 Suppose that $G$ is a bipartite graph on classes $A,B$ and $A$ is finite. Then $$H_n\barr(G[A,B])^1_2$$ for any $n\geq 3$. 

\end{theorem}

Our proof will make use of Theorem \ref{indlemma} i.e. the strong form of $r(H_n,2)=2$ as well as the multi-dimensional  Hales-Jewett theorem \cite{hales} (with dimension $n$, and size of alphabet and number of colours 2) which we state here.

\begin{tlemma}\label{hjlemma}
 Given $\ell\in \NN$ there is $N\in \NN$ so that if the set $\empty^N2$ of all functions from $N$ to $2$ is partitioned as $\mc F_0\cup \mc F_1$ then there is $i<2$, a set $ T=\{t_k:k<\ell\}\subseteq N$ of size $\ell$ and function $h:N\setm T\to 2$  so that   $h\cup g\in \mc F_i$ for any $g:T\to 2$.
\end{tlemma}
 
\begin{proof}[Proof of Theorem \ref{balembed}] Fix $G$ on classes $A,B$. We will show $H_3\barr(G)^1_2$. Then, using Theorem \ref{indlemma} and induction on $n$, the general result follows.


Suppose that $G$ is on classes $A=\{0\}\times \ell$ and $B=\{1\}\times\NN$ where $\ell\in \NN$. Fix a balanced partition $V_0,V_1$ of $H_3$ as well. Our goal is to find $i<2$ and independent $A'\subseteq V_i,B'\subseteq V_{1-i}$ so that $G[A,B]\bij H_3[A', B']$. 

First, given the number $\ell$, the Hales-Jewett theorem provides $N\in \NN$ as in Lemma \ref{hjlemma}. Now, by $r(H_3,2)=2$, there is $X=\{x_0\dots x_{N-1}\}\in [V_0]^N$, $Y=\{y_0\dots y_{N-1}\}\in [V_1]^N$ and $v^*\in V(H_n)\setm (X\cup Y)$ so that $X\cup Y\cup\{v^*\}$ is independent.

Let us define a partition $\mc F_0, \mc F_1$ of $\empty^N2$. We let $f\in \mc F_i$ iff $i<2$ is minimal so that the set $$Z(f,i)=\{v\in V_i:v^*\in N(v), x_k,y_k\in N(v)\text{ if }f(k)=0, x_k,y_k\notin N(v)\text{ if }f(k)=1\}$$ is infinite.  This is well defined by the extension property of $H_3$ and we let $Z(f)=Z(f,i)$ for $f\in \mc F_i$.

\begin{tclaim}\label{indobs}  $Z(f)\cup Z(f')$ is an independent set for any $f,f'\in \ \empty^N2$. 
 \end{tclaim}
\begin{proof}

Indeed, $Z(f)\cup Z(f')\subseteq N(v^*)$ and $N(v^*)$ is independent since $H_3$ is $K_3$-free.
\end{proof}
Now, by the choice of $N$, we can find $i<2$, a set $T=\{t_k:k<\ell\}\subseteq N$ of size $\ell$ and $h:N\setm T\to 2$ so that $h\cup g\in \mc F_i$ for all $g:T\to 2$.

We are ready to define $A'$ and $B'$. Let $A'=\{u_k:k<\ell\}$ where 
\[
    u_k = \begin{cases}
		     x_{t_k}, & \text{if } i=1,\\
        y_{t_k}, & \text{if } i=0.
        \end{cases}\
  \]

Clearly, $A'$ is an independent subset of $V_{1-i}$. To define $B'\subseteq V_i$, we first define functions $g_m:T\to 2$ by letting $g_m(t_k)=0$ iff $(0,k)(1,m)\in E(G)$ for all $m\in \NN$. Let us pick $v_m\in Z(h\cup g_m)$ so that $v_m\neq v_{m'}$ for $m'<m\in \NN$; this can be done as each $Z(h\cup g_m)$  is infinite. We let $B'=\{v_m:m\in \NN\}$ and note that $B'$ is independent by Claim \ref{indobs}. We remark that the only role of $v^*$ was to force $B'$ independent via Claim \ref{indobs}. 

We finish the proof of the theorem by proving the following claim.

\begin{tclaim} The map  $(0,k)\mapsto u_k$ (for $k<\ell$) and $(1,m)\mapsto v_m$ (for $m\in \NN$) witnesses $G[A,B]\bij H_3[A',B']$.
\end{tclaim}

Indeed, $v_m u_k$ is an edge in $H_3$ iff $(h\cup g_m)(t_ k)=0$ iff $g_m(t_k)=0$ iff $(0,k)(1,m)$ is an edge in $G$.

\end{proof}

\section{Finding the exact value of $r(G,m)$ for specific graphs}\label{specificsec}

Next, we present a few further results (and attempts) on finding the exact values of the function $m\mapsto r(G,m)$ for specific $K_n$-free graphs $G$. These examples include shift graphs, unit distance graphs and orthogonality graphs. We begin by a new definition.

\begin{dfn}
 Let $r^*(G,m)=\min\{r:$ if $V_i\in [V]^{|V|}$ for $i<r$ then there is an independent set $A$ so that $|\{i<r:|A\cap V_i|=|V|\}|\geq m \}.$
\end{dfn}

In general, the following holds.

\begin{obs}\label{starobs} Fix any graph $G$.
\begin{enumerate}
            \item  $r(G,m)\leq r^*(G,m)\leq r(G,m+1)-1$ for any $m\geq 2$;
\item  $r^*(G,2)=2$ implies that $r^*(G,m)=r(G,m)=m$ for all $m\geq 2$.

           \end{enumerate}

\end{obs}

\subsection{Shift graphs} Recall that $\textmd{Sh}_n(\kappa)$ denotes the graph on vertices $[\kappa]^n$ so that $pq\in E$ iff $p=\{\xi_0\dots \xi_{n-1}\}$ and $q=\{\xi_1\dots \xi_n\}$ for some increasing sequence $\xi_0<\xi_1<\dots <\xi_n$ from $\kappa$. Our main result on shift graphs is the following.




\begin{theorem}\label{shiftprop}
For all $2\leq n\in \mathbb B$, infinite $\kappa$ and $m<\cf(\kappa)$, $r^*(\textmd{Sh}_n(\kappa),m)=m$.
\end{theorem}

We first need the following form of the well-known $\Delta$-system lemma \cite{kunen}.

\begin{tlemma}\label{deltalemma} Suppose that $\kappa$ is a regular infinite cardinal and $n\in \NN$. If $V$ is a family of $n$-element sets and $V$ has size $\kappa$ then there is a $\Delta$-system $W\subseteq V$ of size $\kappa$ i.e. there is some $r$ (called the \emph{root} of $W$) so that $a\cap b=r$ for all $a\neq b\in W$.
\end{tlemma}

We say that $p,q\subseteq \kappa$ are \textit{strongly disjoint} if $\max(p)<\min(q)$ or $\max(q)<\min(p)$. We prove the theorem now.

\begin{proof}[Proof of Theorem \ref{shiftprop}]

Fix $m<\cf(\kappa),2\leq n\in \NN$ and $V_i\subseteq [\kappa]^n$ of size $\kappa$ for $i<m$.

First, suppose that $\kappa$ is regular and pick $\Delta$-systems $W^i\in [V_i]^\kappa$ with root $r^i$ for $i<m$. By shrinking $W_i$ appropriately, we can suppose that there is a $\delta<\kappa$ so that
\begin{enumerate}
	\item $\sup\{\max (r^i):i<m\}\subseteq \delta$, and
\item  $\{p\setm r^i:p\in W^i, i<m\}$ is strongly disjoint and contained in $\kappa\setm \delta$.
\end{enumerate}
We claim that $A=\bigcup\{W^i:i<m\}$ is the desired independent set. Indeed, if $p\neq q\in A$ then $p\cap q\subseteq \delta$ and $p\setm \delta<q\setm \delta$ or $q\setm \delta<p\setm \delta$. In any case, $pq$ cannot be an edge.

Now, suppose that $\kappa$ is singular. Apply Lemma \ref{deltalemma} to find  $\Delta$-systems $W^i_\varepsilon\subseteq V_i$ of size $\kappa_\vareps$ with root $r^i_\varepsilon$ for each $i<m$ where $(\kappa_\vareps)_{\vareps<\cf(\kappa)}$ is cofinal sequence of regular cardinals in $\kappa$, each bigger than $\cf(\kappa)$. Let $$I=\{i<m:\sup_{\vareps<\cf(\kappa)}(\max (r^i_\vareps))<\kappa\}.$$ We can also suppose that $\sup\bigcup W^i_\vareps<\kappa$ for all $(\vareps,i)\in \cf(\kappa)\times m$. Finally, let $\delta<\kappa$ be an upper bound for all $\sup_{\vareps<\cf(\kappa)}(\max r^i_\vareps)$ where $i\in I$ (this is possible since $|I|<\cf(\kappa)$). 


Our goal is to define $U^i_\vareps\subseteq V_i$ for $(\vareps,i)\in \cf(\kappa)\times m$ by induction on the lexicographical order $<_{\textmd{lex}}$ so that 

\begin{enumerate}
	\item $|U^i_\vareps|=\kappa_\vareps$,
	\item $\sup (\bigcup U^i_\vareps)<\kappa$, 
	\item $a\setm \delta\neq \emptyset$ for all $a\in U^i_\vareps$ and
	\item $ab\notin E(\textmd{Sh}_n(\kappa))$ if $a\in U^i_\vareps,b\in U^j_{\vareps'}$ for $i<j< m$ and $\vareps,\vareps'<\cf(\kappa)$.
 \end{enumerate}

If we succeed then we can find $A_i\subseteq \bigcup\{U^i_\vareps:\vareps<\cf(\kappa)\}$ of size $\kappa$ which is independent (using the Erd\H os-Dushnik-Miller theorem) and then $A=\bigcup \{A_i:i<m\}$ is the desired independent set which meets each $V_i$ in a set of size $\kappa$.

Suppose that $U^i_\vareps$ is defined already for $(\vareps,i)<_{\textmd{lex}}(\vareps^*,j)$. Let $$\lambda=\sup\{\delta,\sup(\bigcup U^i_\vareps):(\vareps,i)<_{\textmd{lex}}(\vareps^*,j)\}$$ and note that $\lambda<\kappa$. 

If $j\in I$ then find $\vareps^*\leq \gamma<\cf(\kappa)$ and $U^j_{\vareps^*}\subseteq W^j_\gamma$ of size $\kappa_{\vareps^*}$ so that $b\setm r^j_\gamma\cap \lambda=\emptyset$ for all $b\in U^j_{\vareps^*}$. Note that if $a\in U^i_\vareps,b\in U^j_{\vareps^*}$ then $a\cap b\subseteq \delta$ and both $a\setm \delta$ and $b\setm \max a$ are not empty. Hence $ab$ is not an edge.

 If $j\in m\setm I$ then find $\vareps^*\leq \gamma<\cf(\kappa)$ and $U^j_{\vareps^*}\subseteq W^j_\gamma$ of size $\kappa_{\vareps^*}$ so that both $r^j_\gamma\setm \lambda$ and $b\setm (r^j_\gamma\cup\lambda)$ are non empty for all $b\in U^j_{\vareps^*}$. Note that if $a\in U^i_\vareps,b\in U^j_{\vareps^*}$ then $2\leq|b\setm \max a|$ and hence $ab$ is not an edge.

\end{proof}

\subsection{Unit distance graphs} Given a metric space $(V,d)$ one defines the \emph{unit distance graph} $G$ corresponding to $(V,d)$ on the vertex set $V$ with $xy\in E(G)$ iff $d(x,y)=1$. 

\begin{prop}
Let $G$ be the unit distance graph of $\mb R^n$ with the usual Euclidean metric. Then $r^*(G,m)=m$ for all $m$.
\end{prop}

Let $\mf c=2^{\aleph_0}$ denote the cardinality of $\mb R$. We say that $x\in \mb R^n$ is a complete accumulation point of  a set $W\subseteq \mb R^n$ if $B\cap W$ has size $|W|$ for any open ball $B$ around $x$.

\begin{proof} It suffices to show that $r^*(G,2)=2$ by Observation \ref{starobs}. Suppose that $V_i\subseteq \mb R^n$ are of size $\mf c$ for $i<2$.  Note that $r^*(G,2)=2$ follows from the claim below.

\begin{tclaim}\label{clm}
 There are complete accumulation points $u_i$ of $V_i$ such that $|u_0-u_1|\neq 1$.
\end{tclaim}

Indeed, if $B_i$ is a small enough ball with radius less than 1 around $u_i$ then $|x-y|\neq 1$ for all $x\in B_i,y\in B_j$ and $i\leq j<2$; hence $A=\bigcup\{B_i\cap V_i:i<2\}$ is the desired independent set.


\begin{proof}[Proof of Claim \ref{clm}] Let  $W\subseteq \mathbb R^n$ be a maximal set of points so that $V_1'=V_1\cap \bigcap\{N_G(u):u\in W\}$ still has size $\mf c$; note that $W$ is finite. Select a complete accumulation point $u_0\in V_0\setm W$ of $V_0$. We claim that $|V_1'\cap N_G(u_0)|<\mf c$. Indeed, otherwise $W'=W\cup \{u_0\}$ still satisfies $$|V_1\cap \bigcap\{N_G(u):u\in W'\}|=\mf c$$ however $W$ was already maximal.



Hence, we can select a complete accumulation point $u_1$ of $V_1'\setm N_G(u_0)$. Now $|u_0-u_1|\neq 1$ so $u_0,u_1$ are as desired.
\end{proof}

And the theorem follows.
\end{proof}

Note that $r^*(G,m)=m$ or even $r(G,m)=m$ can easily fail for other metrics which still induce the Euclidean topology; indeed, if $d(x,y)=\min\{1,|x-y|\}$ then $d$ induces the usual topology while $K_\omega\otimes E_\omega$ embeds into the corresponding unit distance graph. In particular, already $r(G,3)$ and $r^*(G,2)$ does not exist. However, we still have the following:

\begin{prop}
For any metric that induces the usual topology on $\mb R^n$ for $n\geq 2$, the corresponding unit distance graph $G$ will satisfy $r(G,2)=2$.
\end{prop}

The above proposition will be a corollary of the following more general fact.

\begin{lemma} \label{toplemma} Suppose that $G$ is a graph on a separable metric space $(V,d)$. If $V$ has an open cover by $G$-independent sets then either
\begin{enumerate}
 \item $r(G,2)=2$, or
\item there is $Y\subseteq V$ of size $<\mf c$ so that $V\setm Y$ is not connected.
\end{enumerate}

\end{lemma}
\begin{proof} Suppose that (1) fails and this is witnessed by the balanced partition $V_0,V_1$ of $V$. (2) clearly holds if $2\leq |V|<\mf c$ so let us suppose that $|V|=\mf c$.

For every $x\in V$ there is an open neighbourhood $B_x$ of $x$ so that $B_x$ is independent. 
As $B_x$ is independent and (1) fails, there must be a set $Y_x$ of size $<|V|$ and $i_x<2$ so that $B_x\setm Y_x\subseteq V_{i_x}$ for every $x\in V$. Now, there is a countable set $W$ so that $\{B_x:x\in W\}$ covers $V$ so $$V\setm Y= \bigcup \{B_x\setm Y:x\in W\}$$ where $Y=\bigcup \{Y_x:x\in W\}$. Now note that $V_i\setm Y$ is open in $V\setm Y$; indeed, if $z\in V_i\setm Y$ then there is $x\in W$ so that $z\in B_x\setm Y\subseteq V_{i_x}$ and hence $i=i_x$ and $B_x\setm Y$ is an open neighbourhood of $z$ in $V_i\setm Y$. Note that $V_i\setm Y\neq \emptyset$ as $|Y|<|V_i|=\mf c$ (and $c$ has uncountable cofinality). Now, the clopen partition $V\setm Y=(V_0\setm Y)\cup (V_1\setm Y)$ witnesses that $V\setm Y$ is not connected.

\end{proof}

We do need some connectivity assumption, as demonstrated by the next result.

\begin{obs} Suppose that $X\subseteq \mb R^n$ and $\{X_k:k<\ell\}$ is a clopen partition of $X$ into sets of size $|X|$. Then there is a metric $d$ inducing the usual topology on $X$ so that $r(G,2)>\ell$ where $G$ is the unit distance graph on $(X,d)$. 
\end{obs}

In particular, $r(G,2)$ might not exists if $X$ has infinitely many connected components of size $|X|$.

\begin{proof} Simply find a metric $d$ so that the diameter of each $X_k$ is less than 1 while $d(x,y)=1$ if $x\in X_k,y\in X_{k'}$ for some $k<k'<\ell$. The partition $\{X_k:k<\ell\}$ witnesses $r(G,2)>\ell$.
\end{proof}

On the other hand, if  $X\subseteq \mb R^n$ and $\{X_k:k<\ell\}$ is a cover by sets of size $|X|$ which are connected even after the removal of fewer than $\mf c$ points (e.g. $X_k$ is connected and open) then $r(G,2)\leq \ell+1$ by Lemma \ref{toplemma}.

\subsection{Orthogonality graphs} Finally, let us take a look at another class of geometric graphs: let $G_{\mb R^n}$ be defined on vertices $\mb R^n\setm \{0\}$ so that $uv\in E(G_{\mb R^n})$ iff $u\perp v$ i.e. $u$ and $v$ are orthogonal vectors. It is clear that $G_{\mb R^n}$ is $K_{n+1}$-free so $r(G_{\mb R^n},m)$ exists for all $2\leq m\in \NN$.

\begin{prop}
For all $n\geq 2$, $r(G_{\mb R^n},2)=2$. 
\end{prop}
\begin{proof}
 Recall that $\mb R^n\setm Y$ is connected whenever $|Y|<\mf c$. Also, for any $x\neq 0$ there is small open ball around $x$ which is independent in $G_{\mb R^n}$. Hence, Lemma \ref{toplemma} can be applied.
\end{proof}

Unfortunately, finding $r(G_{\mb R^n},m)$ will be much more difficult in general. Let us show first that  finding $r(G_{\mb R^n},m)$ and $r^*(G_{\mb R^n},m)$ will be equally hard. 

\begin{prop}\label{starprop}
For all $n,m\geq 2$, $r(G_{\mb R^n},m)= r^*(G_{\mb R^n},m-1)+1$.
\end{prop}

Let $x^\perp=\{y\in \mathbb R^n:x\perp y\}$ for $x\in \mathbb R^n$ and $B^\perp=\bigcup\{x^\perp:x\in B\}$ for $B\subseteq \mathbb R^n$.

\begin{proof}
 Let us prove first that $r(G_{\mb R^n},m)\leq r=r^*(G_{\mb R^n},m-1)+1$. Take a balanced $r$-partition $\{V_i:i<r\}$; we can suppose that  $V_0$ is dense in some $n$-dimensional ball by the Baire category theorem. Select  $A_i\subseteq V_i$ of size $\mf c$ for $1\leq i<m$ so that $\bigcup\{A_i:1\leq i<r\}$ is independent and let $x_i\in A_i$ be complete accumulation points of $A_i$. Now, it is easy to see that if we take small enough balls $B_i$ around $x_i$ then $V_0\setm \bigcup\{B_i^\perp:i=1\dots m-1\}$ has size $\mf c$. Hence, if $A_0\subseteq V_0\setm \bigcup\{B_i^\perp:i=1\dots m-1\}$ is of size $\mf c$ and independent then $A_0\cup (A_1\cap B_1)\cup\dots \cup(A_{m-1}\cap B_{m-1})$ is the desired independent set.

Equality now follows from Observation \ref{starobs} (1).

\end{proof}

 Now, we characterize $r^*(G_{\mb R^n},m)$ slightly differently.

\begin{obs}
For all natural numbers $n,m\geq2$, $r^*(G_{\mb R^n},m)$ is the minimal number $\hat r=\hat r(n,m)$ such that any $\hat r$ non zero vectors of $\mb R^n$ contain $m$ pairwise non orthogonal points. 
\end{obs}

\begin{proof}
 Let us show $r^*(G_{\mb R^n},m)\leq \hat r$ first: let $V_i\subseteq \mb R^n$ be of size $\mf c$ and pick a complete accumulation point $x_i\in V_i$ for each $i<\hat r$. By the definition of $\hat r$, $\{x_i:i\in I\}$ is pairwise non orthogonal for some set $I\subseteq \hat r$ of size $m$. If $B_i$ is a small enough ball around $x_i$ then $\bigcup\{V_i\cap B_i:i\in I\}$ is the desired independent set.  

On the other hand, take $r^*(G_{\mb R^n},m)$ points $x_i$ and let $V_i$ denote the set of nonzero scalar multiples of $x_i$. Now, if $A\subseteq \bigcup \{V_i:i<r^*(G_{\mb R^n},m)\}$ is the independent set which intersects $m$ of the sets $V_i$ then $\{x_i:|A\cap V_i|\neq \emptyset\}$ must be pairwise non orthogonal. Hence $r^*(G_{\mb R^n},m)\geq \hat r$ holds as well.
\end{proof}

In other words, the largest set $A$ in $\mb R^n$ so that any $B\in [A]^{m+1}$ contains two perpendicular vectors has size $r^*(G_{\mb R^n},m+1)-1$. This number, denoted by $\alpha(n,m)$ was introduced by P. Erd\H os and investigated by several people \cite{rosenfeld, stanley, alon}.  
Let us summarize the known results. Erd\H os conjectured that $\alpha(n,m)=nm$ for all $n,m$ (see \cite{rosenfeld, stanley}) which is translated as $r^*(G_{\mb R^n},m)=n(m-1)+1$. Note that this is true if the points are in general position i.e. any $k$ of them  spans a $k$ dimensional subspace (for $k\leq n$). Indeed, if $A$ is general then we can actually extend any $k$-element pairwise non orthogonal set into an $m$ element pairwise non orthogonal set. To see this, fix $n$ and prove by induction on $m$: fix $k$ points $x_1\dots x_{k}$ which are pairwise non orthogonal; remove $x_i$ and $x_i^\perp\cap A$ from $A$. Note that $|x^\perp\cap A|\leq n-1$ for all $x$ in $A$ hence we still have $n(m-k-1)+1$ points. So, we can select $m-k$ additional vectors which are pairwise non orthogonal using the inductive hypothesis. 

The conjecture in general was disproved by Z. F\"uredi and  R. Stanley \cite{stanley} by showing that there are 24 vectors in $\mb R^4$ without 6 vectors being pairwise non orthogonal. The currently known best lower bound is due to N. Alon and M. Szegedy \cite{alon}: their result shows that there is a constant $\delta>0$ so that $$r^*(G_{\mb R^n},m)> n^{\frac{\delta\log(m+1)}{\log\log(m+1)}}$$ for every large enough $m$ and $n\geq 2\log m$. 

On the other hand, the following upper bound follows from \cite{stanley}:

$$r^*(G_{\mb R^n},m)\leq (1+\text{o}(1))\sqrt{\frac{n\pi}{8}}2^{n/2}(m-1)+1$$ for any $n,m$.

Hence, by Proposition \ref{starprop}, the next corollary holds.

\begin{cor}
  $$n^{\frac{\delta\log(m)}{\log\log(m)}}+1<r(G_{\mb R^n},m) \leq (1+o(1))\sqrt{\frac{n\pi}{8}}2^{n/2}(m-2)+2$$

Here, the lower bound holds for all large enough $m$ and $n\geq 2\log(m)$; the upper bound holds for all $n,m$.
\end{cor}

There is very little known about the exact values of $\alpha(n,m)$ or, equivalently, the values of $r^*(G_{\mb R^n},m)$. Clearly, $r^*(G_{\mb R^n},2)=n+1$. It is easy to see that $r^*(G_{\mb R^2},m)=2(m-1)+1$ and a result of M. Rosenfeld \cite{rosenfeld} yields $r^*(G_{\mb R^n},3)=2n+1$; in particular, the conjecture of Erd\H os still holds for these cases. Now, Proposition \ref{starprop} yields the following.

\begin{cor}
$r(G_{\mb R^2},m)=2(m-1)$, $r(G_{\mb R^n},3)=n+2$ and $r(G_{\mb R^n},4)=2n+2$ for all $n,m$.
\end{cor}

The smallest unknown value to us is $r(G_{\mb R^3},6)$ or equivalently $r^*(G_{\mb R^3},5)$.
\medskip

Finally, we mention that the chromatic number of $G_{\mb R^3}$ is 4  while, somewhat surprisingly, $\chi(G_{\mb R^3}[\mb Q^3])=3$ \cite{orthchrom}. As Corollary \ref{chromcor} can be easily extended to $r^*(G,m)$, we get $r^*(G_{\mb R^3}[\mb Q^3],m)=3(m-1)+1$ as predicted by Erd\H os.

\section{Open problems}\label{probsec}

We close our paper with a list of open problems that we found the most interesting.

\subsection{Questions about $m\mapsto r(G,m)$ in general}

\begin{prob}  Is there a single $K_n$-free graph $G$ so that $r(G,m)=\dr(n,m)$ for all $2\leq m \in \NN$?
\end{prob}

We are not sure how fast $r(G,m)$ might grow for a fixed graph $G$.

\begin{prob}
Suppose that $g:\omega\to \omega$ is monotone increasing. Is there a single graph $G$ so that $g(m)<r(G,m)<\infty$ for all $2\leq m \in \NN$?
\end{prob}

Note that if $g(m)>\dr(n,m)$ for some $m\in \NN$ then  $G$ cannot be $K_n$-free.\\ 

Regarding finite graphs and Theorem \ref{compactness} the obvious question is to determine $N=N(n,m,\ell)$.

\begin{prob}  Estimate/express the function $N=N(n,m,\ell)$ from Theorem \ref{compactness}.
\end{prob}

Next, we mention a question of more set theoretical flavour. The existence of the numbers $r(G,m)$ for a graph $G$ of size $\kappa$ clearly implies that $G$ contains independent sets of size $\kappa$. The same conclusion follows from Hajnal's Set Mapping Theorem \cite{hajnal}: if $\lambda<\kappa$ and each vertex $v$ of a graph $G$ has degree $<\lambda$ then $G$ has an independent set of size $\kappa$. Hence, our question is if one can strengthen Hajnal's theorem as follows.

\begin{prob} Suppose that $\lambda<\kappa$ and each vertex $v$ of a graph $G$ has degree $<\lambda$. Does $r(G,m)=m$ or even $r^*(G,m)=m$ hold for all/some $2\leq m\in \mb N$? 
\end{prob}

Finally, Proposition \ref{findegobs} about countable flat graphs opens the question of calculating $r(G,m)$ for uncountable flat graphs or some subset of them. Flatness is closely related to model-theoretic stability \cite{ziegler} so trying to calculate $r(G,m)$ in classes of model-theoretically tame graphs is another venue likely worth investigating. 

\subsection{Problems on balanced embeddings of graphs}

A natural way to strengthen Theorem \ref{balembed} would be answering the next problem.

\begin{prob}
 Does $H_n\barr(G)^1_2$ hold if $n\geq 3$ and $G$ is an arbitrary subgraph of $\half$? 
\end{prob}

We also ask if the graphs $H_n$ for $n\in \mathbb N$ are essentially the only graphs satisfying Theorem \ref{balembed}.

\begin{prob}
 Characterize those (countable) graphs $H$ so that $H\barr(G[A,B])^1_2$ holds for all finite bipartite $G[A,B]$.
\end{prob}

For example, if $H$ is isomorphic to some $H_n$ modulo a finite set then $H\barr(G[A,B])^1_2$ holds.
\medskip

Now, it would be natural to study balanced embeddings of non bipartite graphs as well. For the simplest case, let us look at $K_3$: suppose that $H$ is a graph so that whenever $V_0\cup V_1\cup V_2$ is a balanced partition of $H$ then there is a copy of $K_3$ with vertices in all dinstinct classes. Let us denote this relation with $H\barr(K_3)^1_3$. 

Any complete graph $H$ satisfies $H\barr(K_3)^1_3$ but not $H_n$; indeed, $H$ cannot contain a copy of $E_{\oo,\oo}$ if $H$ is countable and $H\barr(K_3)^1_3$ and, in turn, any pair of infinte vertex sets $A,B$ is a rich pair. So how sparse can a graph $H$ be while still $H\barr(K_3)^1_3$ holds? For example, it is not hard to see that the uncountable graph $G$ defined at the end of Section \ref{propssec} satisfies $G\barr(K_3)^1_3$.

\begin{prob}
 Characterize those (countable) graphs $H$ so that $H\barr(K_3)^1_3$ holds.
\end{prob}

\subsection{Finding the exact value of $r(G,m)$ for specific graphs}

Finally, it would be interesting to see the exact values of $r(G,m)$ determined for any particular $K_n$-free graphs. 

\begin{prob}
 Let $G_{\mb R^3}$ denote the orthogonality graph on $\mb R^3\setm \{0\}$. Is $r^*(G_{\mb R^3},m)=3(m-1)+1$ for all $m\geq 2$?
\end{prob}

\section{Acknowledgements}

Part of this work was completed while the second and third authors were Postdoctoral Scholars at the University of Calgary supported in part by NSERC of Canada Grant \# 10007490 and PIMS. The third author was also supported in part by the FWF Grant I1921.

%
%
%
%
%

\end{document}